\documentclass[11pt]{article}

\usepackage[margin=1in]{geometry}
\usepackage{amsmath}
\usepackage{amssymb}
\usepackage{amsthm}
\usepackage{booktabs}
\usepackage{color}
\usepackage{hyperref}

\DeclareMathOperator{\tr}{tr}
\DeclareMathOperator{\conv}{conv}
\DeclareMathOperator{\diag}{diag}

\DeclareMathOperator{\rank}{rank}

\newtheorem{theorem}{Theorem}
\newtheorem{corollary}[theorem]{Corollary}
\newtheorem{lemma}[theorem]{Lemma}
\newtheorem{proposition}[theorem]{Proposition}

\newcommand{\Kprod}{\otimes}

\newcommand{\Rbb}{\mathbb{R}}

\newcommand{\tran}{^T}
\newcommand{\norm}[1]{\Vert #1 \Vert}
\newcommand{\leaveout}[1]{}
\newcommand{\set}[1]{\{ #1 \}}
\newcommand{\mprod}[2]{{\langle {#1},{#2}\rangle}}
\newcommand{\suchthat}{\,:\,}
\newcommand{\abs}[1]{\vert #1 \vert}

\newcommand{\gesem}{\succeq}

\newcommand{\abar}{\bar{a}}
\newcommand{\Acal}{{\cal A}}

\newcommand{\Bbar}{\bar{B}}
\newcommand{\eps}{\epsilon}
\newcommand{\lam}{\lambda}
\newcommand{\Ecal}{{\cal E}}
\renewcommand{\hbar}{\bar{h}}
\newcommand{\Hbar}{\bar{H}}
\newcommand{\Ihat}{\hat{I}}
\newcommand{\Kprime}{K^\prime}
\newcommand{\Mbar}{\bar{M}}
\renewcommand{\SS}{J}% letter used for skew-skew matrices
\newcommand{\Ubar}{\bar{U}}
\newcommand{\vbar}{\bar{v}}
\newcommand{\Wcal}{{\cal W}}

\newcommand{\Xbar}{\bar{X}}
\newcommand{\xbar}{\bar{x}}
\newcommand{\ybar}{\bar{y}}

\newcommand{\zstar}{z^*}
\newcommand{\Zbar}{\bar{Z}}

\newcommand{\LOP}{{\rm LOP}}
\newcommand{\SEP}{{\rm SEP}}

\title{Applications of the Lorentz positive cone in \\ nonconvex quadratic optimization}

\author{%
Samuel Burer%
\thanks{Department of Business Analytics, University of Iowa,
Iowa City, IA, 52242-1994, USA. Email: {\tt samuel-burer@uiowa.edu}.}%
\and
Kurt M. Anstreicher%
\thanks{Department of Business Analytics, University of Iowa,
Iowa City, IA, 52242-1994, USA. Email: {\tt kurt-anstreicher@uiowa.edu}.}%
}

\date{\today}

\begin{document}

\maketitle

\begin{abstract}
We consider the Lorentz positive cone of $n\times m$ matrices that
map the Lorentz cone in $\Rbb^m$ into the Lorentz cone in $\Rbb^n$.
The Lorentz positive cone and its dual, the cone of Lorentz separable
matrices, are shown to provide polynomial-time algorithms for the
problem of minimizing a bilinear objective over variables contained in
ellipsoids in $\Rbb^n$ and $\Rbb^m$. We also demonstrate how these cones
can be used to strengthen SDP relaxations of other nonconvex quadratic
optimization problems, including the two-trust-region subproblem.
\end{abstract}

\section{Introduction}

For $n\ge 3$, let $L_n$ denote the Lorentz, or second-order cone, in $\Rbb^n$,
\[
L_n = \set{(y_0,y_1,\ldots,y_{n-1})\tran \suchthat \sum_{i=1}^{n-1} y_i^2 \le y_0^2, y_0 \ge 0}.
\]
The Lorentz positive cone $\LOP(n,m)$ is then the set of $n\times m$ matrices that map $L_m$ into $L_n$,
\[
\LOP(n,m) = \set{M\in \Rbb^{n\times m} \suchthat y=Mx\in L_n\, \forall x\in L_m}.
\]
The Lorentz separable cone  $\SEP(n,m)$ is the set of $n\times m$ matrices that can be
written as the sum of rank-one matrices $yx\tran$ where $y\in L_n$ and $x\in L_m$,
\[
 \SEP(n,m) = \set{M = \sum_{j=1}^k y^j(x^j)\tran \suchthat y^j\in L_n,\,  x^j\in L_m,\,  j=1,\ldots,k}.
\]
It is then easy to show \cite[Corollary 2.6]{Hildebrand.2011} that  $\LOP(n,m)$ is the dual
of $\SEP(n,m)$,
\[
\LOP(n,m) = \SEP^*(n,m) =  \set{S\in\Rbb^{n\times m}\suchthat \mprod{S}{M} \ge 0\, \forall
M \in \SEP(n,m)},
\]
and therefore $\SEP(n,m)= \LOP^*(n,m)$ since $\SEP(n,m)$ is closed. That $\SEP(n,m)\subset \LOP(n,m)$ follows from the fact that $L_m$ is self-dual.   Finally it is clear that by definition, $M\in\SEP(n,m) \iff M\tran\in\SEP(m,n)$, and therefore
$M\in\LOP(n,m) \iff M\tran\in\LOP(m,n)$ as well.

In \cite{Hildebrand.2007}, R. Hildebrand constructed a linear matrix inequality (LMI) that exactly represented $\LOP(n,m)$. The existence of such an LMI had been a long-open question.  The representation in \cite{Hildebrand.2007} was of exponential size, but in subsequent work Hildebrand \cite{Hildebrand.2011} showed that 
$\LOP(n,m)$ can be represented using an LMI of polynomial size in $n$ and $m$.  We describe this representation and the corresponding representation for the
dual cone $\SEP(n,m)$ in Section \ref{sec:cones}.  We also show that there is a 
simple separation procedure for $\LOP(n,m)$ that requires only the solution of an ordinary trust-region subproblem (TRS).  The separation procedure is 
potentially significant because the explicit LMI representation of $\LOP(n,m)$ from \cite{Hildebrand.2011}, while polynomial, is
too large to be computationally practical for $n$ and $m$ beyond approximately 30.

In the setting of 
optimization problems, the ``primal" problem is typically posed over variables that are in $\SEP(n,m)$ and the conic dual problem has
constraints that include $\LOP(n,m)$.  In Section 3 we consider the problem of minimizing a bilinear objective over variables $x\in\Rbb^m$ and $y\in\Rbb^n$ constrained to be in ellipsoids. This problem can be exactly formulated over the $\SEP(n+1,m+1)$ cone and is therefore solvable in polynomial time.  We describe two different solution approaches that reduce the computational burden compared to using exact representations of $\SEP(n+1,m+1)$ or $\LOP(n+1,m+1)$. In particular we show that a dual formulation that combines bisection search and use of the TRS-based separation oracle provides a simple, polynomial-time algorithm without requiring the use of a framework such as the ellipsoid algorithm. 

In Section 4 we consider three more general quadratic problems. In the first case, we add
quadratic terms $x\tran Q x$ and $y\tran P y$ to the bilinear problem from Section 3, so that the objective is a general quadratic in the variables $x$ and $y$.   In the second case, we consider the two-trust-region subproblem (TTRS), where variables $x\in\Ecal_x\subset\Rbb^n$ are also constrained to be in a second ellipsoid,
which can be expressed as  $y=Ax+b\in\Ecal_y\subset\Rbb^n$.  In the third case, we consider a noxious location problem with nonconvex
quadratic constraints.  For all of these problems we show how adding constraints based on the  $\SEP$ cone can be used to tighten the
Shor SDP relaxation.  In computational tests we also compare the improvement obtained using the $\SEP$ cone to another known
methodology based on the Kronecker product of second-order cones \cite{Anstreicher.2017}.

\medskip
\noindent{\bf Notation} All vectors and matrices are real. For $n\times m$ matrices $A$ and $B$, $\mprod{A}{B}$ denotes the inner product $\mprod{A}{B}= \tr(AB\tran)$. The space of $n\times n$ symmetric matrices  is denoted $S^n$, and $n\times n$ positive semidefinite (PSD) matrices are denoted $S^n_+$. We use $(A\, ; B)$ to denote the vertical concatenation $(A\, ; B) = (A\tran, B\tran)\tran$. 
Vectors in $\Rbb^n$ are indexed using $\set{0,1,\ldots, n-1}$, with corresponding indexing for rows and
columns of matrices. We use $e_i$ to denote a vector with a one in the $i$th coordinate and all other elements equal to zero, whose dimension varies with the context.

\section{The Lorentz positive and separable cones} \label{sec:cones}

In this section we describe properties of the Lorentz positive cone $\LOP(n+1,m+1)$ and its dual $\SEP(n+1,m+1)$, where $\min(n,m)\ge 2$.
We use dimensions $(n+1,m+1)$
rather than $(n,m)$ throughout for easier application to the optimization problems considered in the sequel.
 We will first describe the LMI description for $\LOP(n+1,m+1)$ from \cite{Hildebrand.2011}.  There are two main elements in the construction of this LMI.  The first is a linear mapping $\Wcal(\cdot)$ from the set of $(n+1) \times (m+1)$ matrices to the space $S^N$ of symmetric $N\times N$ matrices, where $N=nm$.  To construct this mapping, first consider a mapping
$W_r(\cdot):\Rbb^r \to S^{r-1}$,
\[
W_r(x) = \begin{pmatrix} x_0+x_1 & x_2 & \ldots& x_{r-1}\\
              x_2 & x_0-x_1 & \\
              \vdots & & \ddots \\
              x_{r-1} &  &  & x_0 - x_1 \end{pmatrix}
 \]
The mapping $\Wcal(\cdot): \Rbb^{(n+1)\times(m+1)}\to S^N$ is constructed using $W_{n+1}(\cdot)$ and $W_{m+1}(\cdot)$ via the definition
\[
\Wcal(A) = \sum_{i=0}^{n}\sum_{j=0}^{m} a_{ij}W_{n+1}(e_i)\Kprod W_{m+1}(e_j).
\] 
Next, let $\Acal(r)$ denote the space of skew-symmetric $r\times r$ matrices.
The second component in the LMI description from \cite{Hildebrand.2011} is the subspace of
$S^N$ spanned by Kronecker products of matrices in $ \Acal(n)$ and
$ \Acal(m)$.
 For any $(r,s)$, an orthogonal basis for the subspace $\Acal(r)\Kprod \Acal(s)$ is given by matrices of the form
 \begin{equation}\label{eq:KPSS_basis}
 (E_{ij} - E_{ji})\Kprod(E_{kl} - E_{lk}),\quad 0\le i<j\le r-1,\  0\le k<l\le s-1,
 \end{equation}
where $E_{ij}=e_ie_j\tran$, and therefore $\Acal(r)\Kprod \Acal(s)$ has dimension $rs(r-1)(s-1)/4$.  Each matrix in \eqref{eq:KPSS_basis} has 4 nonzero entries, with Frobenius norm equal to 2. With these definitions we can now give Hildebrand's LMI description for $\LOP(n+1,m+1)$.

\begin{proposition}{\rm \cite{Hildebrand.2011}}.  \label{prop:LOP_LMI}
Let $M\in\Rbb^{(n+1)\times (m+1)}$, $\min(n,m)\ge 2$. Then $M\in \LOP(n+1,m+1)$ if and only if there is an $\SS\in  \Acal(n)\Kprod  \Acal(m)$ so that
$\Wcal(M) + \SS \gesem 0$.
\end{proposition}

Using Proposition \ref{prop:LOP_LMI} we can also give an LMI description for the dual cone $\SEP(n+1,m+1) = \LOP^*(n+1,m+1)$. By definition,
\[
S\in\SEP(n+1,m+1) \iff \mprod{S}{M} \ge 0\ \forall M\in\LOP(n+1,m+1),
\]
which by Proposition \ref{prop:LOP_LMI} is equivalent to 
\begin{equation}\label{eq:SEP_min}
\min\set{\mprod{S}{M}\suchthat \Wcal(M) + \SS \gesem 0, \SS\in  \Acal(n)\Kprod  \Acal(m)} = 0,
\end{equation}
where $M\in\Rbb^{(n+1)\times (m+1)}$. Let $\SS^i$, $i=1,\ldots,k$ be a basis for $ \Acal(n)\Kprod  \Acal(m)$, for example the basis from
\eqref{eq:KPSS_basis}.  Then the optimization problem in \eqref{eq:SEP_min}
is equivalent to
\begin{eqnarray}
&& \min\set{\mprod{S}{M}\suchthat \Wcal(M) + \sum_{i=1}^k \alpha_i \SS^i \gesem 0}\nonumber\\
&=& \max_{T\gesem 0} \min_{M,\alpha}\set{\mprod{S}{M} - \mprod{T}{\Wcal(M) + \sum_{i=1}^k \alpha_i \SS^i}}\nonumber\\
&=& \max_{\substack{T\gesem 0\\ \mprod{T}{\SS^i} = 0\ \forall i}} \min_{M}\set{\mprod{S}{M} - \mprod{T}{\Wcal(M)}}\nonumber\\
&=& \max_{\substack{T\gesem 0\\ \mprod{T}{\SS^i} = 0\ \forall i}} \min_{M}\set{\mprod{S-\Wcal^*(T)}{M}},\label{eq:SEP_min2}
\end{eqnarray}
where $\Wcal^*(\cdot):S^N\to\Rbb^{(n+1)\times (m+1)}$ is the adjoint operator
\[
\Wcal^*(T)_{ij} = \mprod{T}{W_{n+1}(e_i)\Kprod W_{m+1}(e_j)}. 
\]
Since the problem in \eqref{eq:SEP_min} is equivalent to \eqref{eq:SEP_min2}, we immediately obtain the following characterization for $\SEP(n+1,m+1)$.
 
\begin{proposition} \label{prop:SEP_LMI}
Let $S\in\Rbb^{(n+1)\times (m+1)}$, $\min(n,m)\ge 2$. Then $S\in \SEP(n+1,m+1)$ if and only if $S=\Wcal^*(T)$ for $T\gesem 0$ with $\mprod{T}{\SS} = 0\ \forall \SS \in  \Acal(n)\Kprod  \Acal(m)$.
\end{proposition}

The representations in Propositions \ref{prop:LOP_LMI} and \ref{prop:SEP_LMI} both involve the PSD constraint
on a matrix in $S^N$, $N=nm$ and the subspace $ \Acal(n)\Kprod \Acal(m)$ with $O(n^2m^2)$ generators. While polynomial, the size of these representations quickly becomes prohibitively large for efficient computation. 

\subsection{A separation oracle for the Lorentz positive cone}

One approach to reduce the computational
burden for an optimization problem that includes $\LOP(n+1,m+1)$ is to not explicitly enforce the constraint that $M\in\LOP(n+1,m+1)$ but instead use valid inequalities of the form $\mprod{A}{M}\, \ge 0$ which hold for any $M\in\LOP(n+1,m+1)$. If the current iterate $M_k$ is not in $\LOP(n+1,m+1)$ then $M_k$ is separated from $\LOP(n+1,m+1)$ using a new valid inequality with $\mprod{A_k}{M_k}\,<0$.
Such a cut or separating hyperplane is then added to the problem description.  If the problem of finding a separating hyperplane can be solved in polynomial time, such a
separation oracle can be combined with the ellipsoid algorithm \cite{Grotschel.Lovasz.Schrijver.1988} or volumetric cutting-plane method \cite{Anstreicher.1997,Vaidya.1996} to provide a polynomial-time algorithm for
the original problem.  More generally, the use of such cuts in an outer approximation method can be very efficient computationally even if not theoretically polynomial-time.

We next describe an efficient separation oracle for $ \LOP(n+1,m+1)$. The existence of such a separation oracle was first suggested by Nemirovski \cite{Nemirovski.2005} and is described in more detail in \cite[Section 3.2.5]{Nemirovski.2012}. We provide an independent separation oracle. By definition, 
\[
M\in \LOP(n+1,m+1) \iff Mx \in L_{n+1}\ \forall x\in L_{m+1}.
\]
For $x\in\Rbb^{m+1}$, $y\in\Rbb^{n+1}$ and $M\in\Rbb^{(n+1)\times (m+1)}$ it will be convenient to
write
\[
x=\begin{pmatrix} x_0 \\ \xbar \end{pmatrix},  \quad
y=\begin{pmatrix} y_0 \\ \ybar \end{pmatrix}, \quad
M=\begin{pmatrix} M_0\tran \\ \Mbar \end{pmatrix}.
\]
If there is an $x$ in $L_{m+1}$ with $M_0\tran x < 0$ then clearly $M\notin \LOP(n+1,m+1)$. Since $L_{m+1}$
is self-dual, this possibility can be avoided by imposing the convex constraint that $M_0\tran\in L_{m+1}$.  Assuming this is the case, we then have
\begin{eqnarray*}
M\in \LOP(n+1,m+1) &\iff& \norm{\Mbar x}\le M_0\tran x\ \forall x\in L_{m+1} \\
&\iff& \min\set{(M_0\tran x)^2 - \norm{\Mbar x}^2 \suchthat x\in L_{m+1}} \ge 0 \\
&\iff& \min\set{(M_0\tran x)^2 - \norm{\Mbar x}^2 \suchthat \norm{\xbar}\le x_0=1} \ge 0.
\end{eqnarray*}
The last problem is an ordinary trust-region subproblem (TRS) in the variables $\xbar\in\Rbb^{m}$ which can be solved in polynomial time \cite{Fu.Luo.Ye.1998}. Suppose
now that an $\xbar$ with $\norm{\xbar}\le x_0 = 1$ has $Mx = v$ with $v_0=M_0\tran x < \norm{\vbar} = \norm{\Mbar x}$. Note that $M_0\tran x \ge 0$
implies that $\vbar\ne 0$. Let $\ybar = -\vbar/\norm{\vbar}$, $y = (1\, ; \ybar)$ and $A= yx\tran$.
Then $y \in L_{n+1}$, so $\mprod{A}{S} = y\tran S x \ge 0$ for any $S\in \LOP(n+1,m+1)$. However
\begin{equation}\label{eq:separation}
\mprod{A}{M} = y\tran Mx =
\begin{pmatrix} 1 \\ -\vbar/\norm{\vbar} \end{pmatrix}\tran \begin{pmatrix} M_0\tran x \\ \Mbar x  \end{pmatrix}
=v_0 - \norm{\vbar} <0,
\end{equation}
so $\mprod{A}{S} \ge 0$ is a valid linear constraint that separates $M$ from $ \LOP(n+1,m+1)$.

The TRS required in the separation oracle is a problem of the form $\min\set{\xbar\tran \Hbar\xbar +2\hbar\tran \xbar \suchthat \norm{\xbar}\le 1}$.  One approach to
solving this problem is to consider the semidefinite programming (SDP) problem $\min\set{\mprod{H}{X} \suchthat X\gesem 0,\, \tr(\Xbar) \le 1}$, where
\[
X= \begin{pmatrix} 1 & \xbar\tran \\ \xbar &\Xbar\end{pmatrix},\quad H = \begin{pmatrix} 0 & \hbar\tran \\\hbar & \Hbar\end{pmatrix}.
\]
It is well known that this SDP is an exact representation of the original TRS \cite{Rendl.Wolkowicz.1997}.  Computationally, the
solution of the SDP may fail to be rank one ($\Xbar = \xbar\xbar\tran)$  if the TRS has multiple optimal solutions.  In this case there are a variety of different methods for
recovering a solution $\xbar$.  One simple approach is to add the constraint $\mprod{H}{X} = \zstar$, where $\zstar$ is the solution value for the SDP, and replace the objective matrix $H$ with a matrix $C$ having random components.  With probability one, re-solving the SDP will then generate a rank-one solution $\Xbar=\xbar\xbar\tran$ with $\mprod{H}{X}=\xbar\tran \Hbar\xbar +2\hbar\tran \xbar =\zstar$. 

\subsection{Properties of the Lorentz separable cone}

The Lorentz separable cone $\SEP(n+1,m+1)= \LOP^*(n+1,m+1)$ consists of matrices of the form
\begin{equation}\label{eq:SEP_decomp}
 \sum_{i=1}^k \begin{pmatrix} \beta_i \\ v^i\end{pmatrix}
  \begin{pmatrix} \alpha_i \\ u^i\end{pmatrix}\tran,
  \end{equation}
  where $v^i\in\Rbb^{n}$, $\norm{v^i}\le\beta_i$, 
  $u^i\in\Rbb^{m}$, $\norm{u^i}\le\alpha_i$, $ i=1,\ldots,k$.
Any nonzero element $Z\in \SEP(n+1,m+1)$ can be normalized by setting $Z_{00}=1$, and we will repeatedly use this normalization in the sequel. The next lemma shows that any $Z\in \SEP(n+1,m+1)$ with $Z_{00}=1$
can be written as the convex combination of rank-one matrices of the form $(1\, ; y)(1\, ; x)\tran$ with
 $\norm{x}\le 1$ and $\norm{y}\le 1$.

\begin{lemma}\label{lem:Z_decomp}
Suppose that $Z\in \SEP(n+1,m+1)$ with $Z_{00} = 1$. Then there are $y^i\in\Rbb^{n}$,
$x^i\in\Rbb^{m}$ and $\lam_i>0$, $i=1,\ldots,k$ with $\norm{y^i}\le 1$,  $\norm{x^i}\le 1$, and
$\sum_{i=1}^k \lam_i=1$  such that
\[
Z= \sum_{i=1}^k \lam_i \begin{pmatrix} 1\\ y^i\end{pmatrix}
  \begin{pmatrix} 1 \\ x^i\end{pmatrix}\tran.
  \]
  \end{lemma}
  
\begin{proof}
Since $Z\in\SEP(n+1,m+1)$, $Z$ has a decomposition \eqref{eq:SEP_decomp} with $\sum_{i=1}^k \alpha_i\beta_i = 1$. Since $\alpha_i = 0$ implies that $u^i=0$, and $\beta_i=0$ implies that
$v^i=0$, we may assume that $\alpha_i>0$ and $\beta_i>0$ for each $i$.  The result follows by setting $y^i=v^i/\beta_i$, $x^i=u^i/\alpha_i$ and $\lam_i=\alpha_i\beta_i$ for each $i$. 
\end{proof}

Next assume that $Z\in\SEP(n+1,m+1)$ has the form
\begin{equation}\label{eq:Zform}
Z = \begin{pmatrix} 1 & x\tran \\y & V \end{pmatrix},\ y\in\Rbb^{n},\  x\in\Rbb^{m}.
\end{equation}

\begin{lemma}\label{lem:nuclear_norm}
Suppose that $Z\in\SEP(n+1,m+1)$ as in \eqref{eq:Zform}. Then $\norm{V-yx\tran}_*\le 
\sqrt{(1-\norm{x}^2) (1-\norm{y}^2)}$, where $\norm{V}_*$ denotes the nuclear norm of $V$.
\end{lemma}

\begin{proof}
Using Lemma \ref{lem:Z_decomp}, we have $y=\sum_{i=1}^k \lam_i y^i$, $x=\sum_{i=1}^k \lam_i x^i$ and $V=\sum_{i=1}^k \lam_i  (y^i)(x^i)\tran$,
where $\norm{y^i}\le 1$ and $\norm{x^i}\le 1$ for each $i$. It follows that
$V-yx\tran = \sum_{i=1}^k \lam_i (y^i-y)(x^i-x)\tran$. Therefore 
\begin{eqnarray*}
\norm{V-yx\tran}_* &\le& \sum_{i=1}^k \lam_i \norm{(y^i-y)(x^i-x)\tran}_*\\
&=&\sum_{i=1}^k \lam_i \norm{y^i-y}\norm{x^i-x}\\
&=& a\tran b,
\end{eqnarray*}
where $a_i=\sqrt{\lam_i}\norm{x^i-x}$ and $b_i=\sqrt{\lam_i}\norm{y^i-y}$, $i=1,\ldots,k$. But
\[
a\tran b\le \norm{a} \norm{b}
= \left(\sum_{i=1}^k \lam_i\norm{y^i-y}^2\right)^\frac{1}{2}\left(\sum_{i=1}^k \lam_i\norm{x^i-x}^2\right)^\frac{1}{2}.
\]
Finally,
\begin{eqnarray*}
\sum_{i=1}^k \lam_i\norm{x^i-x}^2 &=&
\sum_{i=1}^k \lam_i(\norm{x^i}^2 -2 x\tran x^i + \norm{x}^2) \\
&=& \sum_{i=1}^k \lam_i \norm{x^i}^2 - \norm{x}^2\\
&\le& 1 - \norm{x}^2,
\end{eqnarray*}
and similarly $\sum_{i=1}^k \lam_i\norm{y^i-y}^2 \le 1- \norm{y}^2$.
\end{proof}

\begin{corollary}\label{cor:SEP_boundary}
Suppose that $Z\in\SEP(n+1,m+1)$ as in \eqref{eq:Zform}, and $\norm{V-yx\tran}_* = 
\sqrt{(1-\norm{x}^2) (1-\norm{y}^2)}$. Then $Z$ is on the boundary of $\SEP(n+1,m+1)$.
\end{corollary}

\begin{proof}
Lemma \ref{lem:nuclear_norm} implies that any $Z\in\SEP(n+1,m+1)$\ as in \eqref{eq:Zform} also has $\norm{V-yx\tran}_F \le 
\sqrt{(1-\norm{x}^2) (1-\norm{y}^2)}$, since $\norm{\cdot}_F \le \norm{\cdot}_*$.  But if 
$\norm{V-yx\tran}_F  = \sqrt{(1-\norm{x}^2) (1-\norm{y}^2)}$, it is obvious that we can make arbitrarily small changes to
$V$, without changing $x$ and $y$, and violate the Frobenius-norm inequality. Therefore $Z$ lies on the boundary of $\SEP(n+1,m+1)$.
\end{proof}

\section{Bilinear optimization over ellipsoids}

In this section we consider the bilinear optimization problem
\begin{eqnarray}
&\min& c\tran x + d\tran y + y\tran R x\label{eq:bilinear}\\
&{\rm s.t.}& x\in \Ecal_x\subset\Rbb^m,\quad y\in \Ecal_y\subset\Rbb^n,\nonumber
\end{eqnarray}
where $\Ecal_x$ and $\Ecal_y$ are both full-dimensional ellipsoids.
After an affine change of variables we can assume without loss of generality that $\Ecal_x = \set{x\in\Rbb^{m} \suchthat \norm{x}\le 1}$ and $\Ecal_y = \set{y\in\Rbb^{n} \suchthat \norm{y}\le 1}$, so
\[
 \begin{pmatrix} 1 \\ x \end{pmatrix} \in L_{m+1},\quad 
\begin{pmatrix} 1 \\ y \end{pmatrix} \in L_{n+1}.
\]
Using the augmented cost matrix
\[
C = \begin{pmatrix} 0 & c\tran \\ d & R \end{pmatrix}
\]
we can then re-write \eqref{eq:bilinear} as
\begin{equation}\label{eq:bilinear_SEP}
\min \set{\mprod{C}{Z} \suchthat Z\in \SEP(n+1,m+1),\ Z_{00} = 1}.
\end{equation}
The representation in \eqref{eq:bilinear_SEP} is exact due to the decomposition guaranteed to exist for any  $Z\in \SEP(n+1,m+1)$
with  $Z_{00} = 1$ from Lemma \ref{lem:Z_decomp}.
The dual of the conic problem \eqref{eq:bilinear_SEP} is 
\begin{equation}\label{eq:bilinear_dual1}
\max\set{\alpha\suchthat  C-\alpha E_{00}\in \LOP(n+1,m+1)},
\end{equation}
which using the result of Proposition \ref{prop:LOP_LMI} is equivalent to the problem
\[
\max\set{\alpha\suchthat  \Wcal(C-\alpha E_{00}) + \SS \gesem 0, \SS\in\Acal(n)\Kprod\Acal(m)}.
\]
Note that $\Wcal(E_{00}) = I\Kprod I$, so the dual problem can be written
\begin{equation}\label{eq:bilinear_dual2}
\max\set{\alpha\suchthat  \Wcal(C)-\alpha I + \SS \gesem 0, \SS\in\Acal(n)\Kprod\Acal(m)}.
\end{equation}
Taking the dual of the above problem, we obtain an explicit, convex representation for the original problem \eqref{eq:bilinear},
\begin{eqnarray}
&\min& \mprod{\Wcal(C)}{S} \nonumber \\
&{\rm s.t.}& \tr(S) = 1 \label{eq:Full_SEP}\\
&& \mprod{\SS}{S}  = 0\ \forall \SS \in \Acal(n)\Kprod \Acal(m) \nonumber\\
&& S\gesem 0, \nonumber
\end{eqnarray}
where the matrix $S$ is $nm\times nm$. We will refer to the problem in \eqref{eq:Full_SEP} as the ``Full SEP" problem since
it is equivalent to the formulation of the bilinear problem over $\SEP(n+1,m+1)$ in \eqref{eq:bilinear_SEP}. 

Next we will describe three different approaches to solving the original problem \eqref{eq:bilinear} using either the 
$\SEP(n+1,m+1)$ or
$\LOP(n+1,m+1)$ cones.  The first is to directly solve the Full SEP problem in \eqref{eq:Full_SEP}.  Note that from \eqref{eq:KPSS_basis},
the constraint that $\mprod{\SS}{S} = 0\ \forall \SS \in \Acal(n)\Kprod \Acal(m)$ requires $nm(n-1)(m-1)/4$ equations to fully represent.  An alternative to adding all of these constraints is to enforce them on an ``as needed" basis by first solving \eqref{eq:Full_SEP} with these constraints omitted and then checking if any are violated beyond a prescribed tolerance for zero.  If violated equalities are found then they are added to the problem and it is re-solved.  This process is then repeated until a solution is generated that violates none of the equalities.  We refer to this algorithmic approach as ``Lazy SEP." 

Lastly we consider an approach to \eqref{eq:bilinear} based on the dual problem \eqref{eq:bilinear_dual1}.  Note that the dual problem
has only a single variable $\alpha$. Rather than use the explicit representation \eqref{eq:bilinear_dual2}, we will apply the separation
oracle described in the previous section in conjunction with a bisection search on $\alpha$.  We are interested in finding the largest value of $\alpha$ satisfying the condition that
$M(\alpha)\in\LOP(n+1,m+1)$ or equivalently $\norm{\Mbar(\alpha) x}\le M_0\tran(\alpha) x\ \forall x\in L_{m+1}$, where
\[
M(\alpha) =\begin{pmatrix} M_0\tran(\alpha) \\ \Mbar(\alpha) \end{pmatrix}
= C - \alpha E_{00} = \begin{pmatrix} -\alpha & c\tran \\ d & R \end{pmatrix}.
\]
The condition that $M_0\tran(\alpha) \in L_{m+1}$ is simply  $-\alpha\ge\norm{c}$ or $\alpha\le -\norm{c}$.  The fact that $M\tran(\alpha)\in\LOP(m+1,n+1)$ 
similarly implies that we must have $\alpha\le -\norm{d}$.
On the other hand, if $\alpha\le -\norm{c} - \norm{d} - \norm{R}_2$, then for any $x = (1 ; \xbar) \in L_{m+1}$,
\[
M_0\tran(\alpha) x - \norm{\Mbar(\alpha) x} = -\alpha + c\tran\xbar - \norm{d+R\xbar} \ge -\alpha -\norm{c} - \norm{d} -\norm{R}_2 \ge 0,
\]
so an initial interval for the solution value of $\alpha$ can be taken to be $[\alpha_{\min},\alpha_{\max}]$ where $\alpha_{\min} =  -\norm{c} - \norm{d} - \norm{R}_2$, $\alpha_{\max} = -\max(\norm{c},\norm{d})$. For a trial value $\alpha=(\alpha_{\min}+\alpha_{\max})/2$ we can use the separation oracle to
check if $M(\alpha)\in \LOP(n+1,m+1)$. If so we set $\alpha_{\min} = \alpha$, and if not we set $\alpha_{\max} = \alpha$.  This process is repeated until $\alpha_{\max} - \alpha_{\min}\le \eps$, where $\eps$ is a prescribed tolerance for optimality.  

In this application of the separation oracle for $\LOP(n+1,m+1)$ we do not require the separating hyperplane that is found when $M(\alpha)\notin\LOP(n+1,m+1)$, but we
do ultimately need to recover an approximately optimal solution of the original problem \eqref{eq:bilinear}. To this end, suppose that $\alpha\in [\alpha_{\min},
\alpha_{\max}]$ where $\alpha_{\max} - \alpha_{\min} \le \eps$ and the separation oracle finds that $M(\alpha)\notin \LOP(n+1,m+1)$. Then there are $x = (1; \xbar)\in L_{m+1}$ and $y = (1; \ybar)\in L_{n+1}$ with $y\tran M(\alpha) x < 0$, as in \eqref{eq:separation}.  But 
\[
y\tran M(\alpha) x = \begin{pmatrix} 1 \\ \ybar \end{pmatrix}\tran
\begin{pmatrix} -\alpha & c\tran \\ d & R \end{pmatrix}
\begin{pmatrix} 1 \\ \xbar \end{pmatrix} = -\alpha + c\tran \xbar +d\tran\ybar +\ybar\tran R \xbar,
\]
so $y\tran M(\alpha) x < 0$ implies that  $c\tran\xbar +d\tran\ybar +\ybar\tran R \xbar < \alpha \le \alpha_{\min}+\eps$.  Since $\alpha_{\min}$ is a lower bound on the solution value of \eqref{eq:bilinear}, $(\xbar,\ybar)$ is a feasible solution to the problem with objective value within $\eps$ of optimality.

\subsection{Computational results}

We now consider solving instances of the bilinear problem \eqref{eq:bilinear} using the different computational approaches described above. For comparison we also consider 
using the widely-used Gurobi solver, which has the ability to solve nonconvex quadratic problems using convex envelopes and spatial branching. All  runs were done on a Mac mini with Apple M4 chip, 10 cores, and 16 GB memory, running macOS 26.5.2. The conic models were solved with MOSEK 11.2.2 through the Python Fusion API and Gurobi was version 13.0.2 through gurobipy. The scripts used Python 3.14.6, NumPy 2.5.1, and SciPy 1.18.0. Solver defaults were used unless otherwise stated; in particular, we did not manually fix thread counts. For reproducibility, all code is archived at \url{https://github.com/sburer/lorentz-positive}.

\begin{table}
\centering
\caption{Median wall-clock times for bilinear instances.}\label{tab:bilinear_times}
\vskip 10pt
\begingroup
\small
\begin{tabular}{c r  r  r  r  r  r r}
\toprule
$n\times m$ & cases & Shor & Full SEP & Lazy SEP & LOP/TRS & Gurobi & G\_opt \\
\midrule
$2\times 2$ & 10 & 1 ms & 1 ms        & 1 ms & 18 ms & 25 ms & 10 \\
$4\times 4$ & 10 & 1 ms & 2 ms         & 5 ms & 23 ms & 1.81 s & 9 \\
$6\times 6$ & 10 & 1 ms & 11 ms       & 21 ms & 24 ms & 5.00 s & 0 \\
$8\times 8$ & 10 & 2 ms & 49 ms       & 82 ms & 28 ms & 5.00 s & 0 \\
$10\times 10$ & 10 & 2 ms & 321 ms  & 318 ms & 31 ms & 5.01 s & 0 \\
$15\times 15$ & 10 & 3 ms & 12.00 s  & 3.10 s & 40 ms & 36.00 s & 0 \\
$20\times 20$ & 10 & 5 ms & 481.33 s & 51.20 s & 60 ms & 1444.06 s & 0 \\
\midrule
$4\times 8$ & 10 & 1 ms & 7 ms        & 16 ms & 27 ms & 5.00 s & 0 \\
$8\times 12$ & 10 & 2 ms & 248 ms   & 267 ms & 33 ms & 5.01 s & 0 \\
$10\times 20$ & 10 & 3 ms & 7.14 s   & 2.24 s & 51 ms & 21.42 s & 0 \\
$15\times 20$ & 10 & 4 ms & 59.18 s  & 15.22 s & 54 ms & 177.55 s & 0 \\
$15\times 25$ & 10 & 5 ms & 249.54 s & 24.95 s & 65 ms & 748.66 s & 0 \\
\bottomrule
\end{tabular}
\endgroup
\end{table}

We generated 120 random instances of \eqref{eq:bilinear} with
\(\Ecal_x=\{x\in\Rbb^m\suchthat \norm{x}\le 1\}\) and
\(\Ecal_y=\{y\in\Rbb^n\suchthat \norm{y}\le 1\}\). 
For each candidate instance, \(R\in\Rbb^{n\times m}\) was sampled with independent standard normal entries.  Let \(u_1\in\Rbb^n\) and \(v_1\in\Rbb^m\) be the leading left and right singular vectors of \(R\).  We then set \(c=\alpha v_1+\sigma g\) and \(d=\alpha u_1+\sigma h\), where \(g\) and \(h\) have independent standard normal entries, \(\alpha=3.0\), and \(\sigma=0.1\).  Finally, the triple \((c,d,R)\) was divided by \(\max\{\norm{c},\norm{d}\}\), so that \(\max\{\norm{c},\norm{d}\}=1\). As described above this normalization ensures that $\zstar\le -1$, where $\zstar$ is the true solution value of the instance.
For each size we scanned seeds starting from zero and retained the first ten candidates for which the Shor relative gap \((z^*-z_{\rm Shor})/\abs{z^*}\) exceeded \(0.01\). (The Shor
relaxation for \eqref{eq:bilinear} is a special case of the Shor relaxation for the problem considered in Section 4.1; see that section for details.) We compared times for the Shor SDP relaxation, Full SEP, Lazy SEP,
the Dual LOP/TRS bisection method and Gurobi.
Gurobi was run after the three conic methods with explicit bounds
\(-1\le x_j\le 1\), \(-1\le y_i\le 1\) for each $i$ and $j$, the two ball constraints, and time limit $\max\{5,3\,t_{\max}\}$,
where \(t_{\max}\) is the maximum time used by the three conic methods on that
instance, and relative MIP gap tolerance \(10^{-6}\).  Lazy SEP added at most 50 violated equalities in each
outer-approximation round, using relative violation tolerance \(10^{-7}\)
(scaled by \(\norm{S}_F\)).  The Dual
LOP/TRS method used relative bisection tolerance \(10^{-6}\).

The median times required by different methods on these test instances are given in Table \ref{tab:bilinear_times}. In the table, the column $n\times m$ gives the dimensions of $y$ and $x$, respectively. Ten cases were run for each size. The Shor, Full SEP, Lazy SEP, Dual LOP/TRS, and Gurobi columns give median wall-clock times for the five methods.  The G\_opt column gives the number of instances solved to global optimality by Gurobi within its adaptive time limit. 

In Table \ref{tab:bilinear_diagnostics} we give additional numerical
diagnostics for the same instances as in Table \ref{tab:bilinear_times}.
By construction none of the Shor instances has a lower bound that agrees
with the exact conic value within relative tolerance $10^{-2}$. The
``Shor max gap" column gives the largest relative gap between the true
solution and the Shor lower bound. The ``Conic max diff" column is
the largest relative objective difference between the Full SEP value
and the Lazy SEP and Dual LOP/TRS values over all instances of that
size. The Lazy SEP columns give the median/maximum outer-approximation
rounds, total added violated equalities and the largest final relative
equation violation (scaled as in the stopping tolerance). The LOP/TRS
columns report the median/maximum TRS separation oracle solves for the
Dual LOP/TRS method as well as the maximum final relative bisection
bracket width.

\begin{table}[htbp]
\centering
\caption{Additional diagnostics for bilinear instances}\label{tab:bilinear_diagnostics}
\vskip 10pt
\begingroup
\small
\begin{tabular}{c |c| c| r r r| r r}
\toprule
& Shor & Conic & \multicolumn{3}{c|}{Lazy SEP} & \multicolumn{2}{c}{LOP/TRS} \\
$n\times m$ & max gap & max diff & rounds & cuts & max viol & oracles & max gap \\
\midrule
$2\times 2$ & $2.9\cdot 10^{-1}$ & $5.5\cdot 10^{-8}$ & 2/2 & 1/1 & $2.3\cdot 10^{-14}$ & 23/23 & $8.9\cdot 10^{-7}$ \\
$4\times 4$ & $2.4\cdot 10^{-1}$ & $3.2\cdot 10^{-8}$ & 2/2 & 36/36 & $5.8\cdot 10^{-12}$ & 23/23 & $8.3\cdot 10^{-7}$ \\
$6\times 6$ & $1.9\cdot 10^{-1}$ & $1.2\cdot 10^{-7}$ & 3/4 & 100/150 & $1.3\cdot 10^{-8}$ & 23/23 & $8.1\cdot 10^{-7}$ \\
$8\times 8$ & $1.4\cdot 10^{-1}$ & $3.2\cdot 10^{-8}$ & 5/6 & 200/250 & $2.2\cdot 10^{-9}$ & 23/23 & $7.5\cdot 10^{-7}$ \\
$10\times 10$ & $1.2\cdot 10^{-1}$ & $4.1\cdot 10^{-8}$ & 9/12 & 400/550 & $2.1\cdot 10^{-9}$ & 23/23 & $7.3\cdot 10^{-7}$ \\
$15\times 15$ & $6.8\cdot 10^{-2}$ & $1.2\cdot 10^{-8}$ & 22/37 & 1025/1800 & $3.0\cdot 10^{-9}$ & 23/23 & $6.7\cdot 10^{-7}$ \\
$20\times 20$ & $6.7\cdot 10^{-2}$ & $5.5\cdot 10^{-9}$ & 63/76 & 3100/3750 & $9.9\cdot 10^{-9}$ & 23/23 & $6.6\cdot 10^{-7}$ \\
\midrule
$4\times 8$ & $1.8\cdot 10^{-1}$ & $3.2\cdot 10^{-7}$ & 3/4 & 100/150 & $1.8\cdot 10^{-9}$ & 23/23 & $7.9\cdot 10^{-7}$ \\
$8\times 12$ & $1.1\cdot 10^{-1}$ & $1.2\cdot 10^{-7}$ & 9/11 & 400/500 & $2.2\cdot 10^{-9}$ & 23/23 & $7.1\cdot 10^{-7}$ \\
$10\times 20$ & $8.7\cdot 10^{-2}$ & $8.5\cdot 10^{-9}$ & 20/30 & 950/1450 & $2.2\cdot 10^{-9}$ & 23/23 & $6.8\cdot 10^{-7}$ \\
$15\times 20$ & $6.0\cdot 10^{-2}$ & $2.3\cdot 10^{-8}$ & 41/54 & 2000/2650 & $1.9\cdot 10^{-9}$ & 23/23 & $6.5\cdot 10^{-7}$ \\
$15\times 25$ & $5.5\cdot 10^{-2}$ & $3.1\cdot 10^{-7}$ & 46/82 & 2250/4050 & $2.3\cdot 10^{-9}$ & 23/23 & $6.5\cdot 10^{-7}$ \\
\bottomrule
\end{tabular}
\endgroup
\end{table}

To summarize these results, Lazy SEP is slower than Full SEP on smaller instances due to the overhead of checking for violated constraints and repeated conic solves, but the time for Full SEP blows up faster as problem size increases. The Dual LOP/TRS bisection algorithm is very fast and robust and scales much better than Full SEP or Lazy SEP.  These problems are difficult for Gurobi, even with the explicit variable bounds added, and cannot be solved to global optimality in time competitive with the conic methods.

\section{More general quadratic problems}

In this section we consider three different problems with more general quadratic objectives and constraints. In each case our goal is to use additional constraints based on LOP or SEP cones to strengthen the Shor relaxation. In the first case, we 
generalize the bilinear problem from the previous section by adding quadratic terms $x\tran Qx$ and $y\tran Py$ to the objective.  In the second case,
we consider the two-trust-region subproblem (TTRS), where variables $x\in\Rbb^n$ are constrained to lie in two ellipsoids.  In the third case, we consider a problem with nonconvex quadratic constraints which must first be convexified before additional constraints based on the SEP cone can be applied.

\subsection{Quadratic optimization with two ellipsoids}

We first consider the problem 
\begin{eqnarray}
&\min& c\tran x + d\tran y + x\tran Qx + y\tran Py + y\tran Rx\label{eq:quadratic}\\
&{\rm s.t.}& x\in \Ecal_x\subset\Rbb^m,\quad y\in \Ecal_y\subset\Rbb^n,\nonumber
\end{eqnarray}
where $\Ecal_x$ and $\Ecal_y$ are full-dimensional ellipsoids. As for the bilinear
problem \eqref{eq:bilinear}, we can make an affine change of variables and assume without loss of generality that $\Ecal_x = \set{x\in\Rbb^{m} \suchthat \norm{x}\le 1}$ and $\Ecal_y = \set{y\in\Rbb^{n} \suchthat \norm{y}\le 1}$. 
To describe the Shor relaxation of \eqref{eq:quadratic} it is convenient to define
\begin{equation}\label{eq:Shor_matrices}
C = \frac{1}{2} \begin{pmatrix} 0       & c\tran &   d\tran \\
                               c  & 2Q	&R\tran \\
                               d & R & 2P \end{pmatrix},\quad                             
   U =   \begin{pmatrix} 1       & x\tran &   y\tran \\
                               x  & X	&V\tran \\
                               y & V & Y \end{pmatrix}.
 \end{equation}
The Shor relaxation of \eqref{eq:quadratic} is then the SDP problem
\begin{eqnarray}
&\min& \mprod{C}{U}\label{eq:Shor_quadratic}\\
&{\rm s.t.}& \tr{X} \le 1,\ \tr{Y}\le 1,\nonumber\\
& &U\gesem 0,\ U_{00}=1.\nonumber
\end{eqnarray}
Note that the Shor relaxation of the bilinear problem \eqref{eq:bilinear} is simply \eqref{eq:Shor_quadratic} with $Q=0$ and $P=0$.
A rank-one optimal solution of \eqref{eq:Shor_quadratic} immediately provides a solution $(x,y)$ of \eqref{eq:quadratic}, and it follows easily from the Pataki rank bound \cite{Pataki.1998} that if $U$ is an extreme point of \eqref{eq:Shor_quadratic} then $\rank(U) \le 2$. In the latter rank-2 case at optimality, we will consider adding the valid constraint
$Z\in\SEP(n+1,m+1)$, where $Z$ is exactly as in \eqref{eq:Zform}. In the next lemma we show that if $\rank(U)=2$, then the corresponding $Z$ cannot
be in the interior of $\SEP(n+1,m+1)$.
 
\begin{lemma} \label{lem:SEP_interior}
Suppose that $U$ as in \eqref{eq:Shor_matrices} is an extreme point of \eqref{eq:Shor_quadratic} with $\rank(U) = 2$.  Let $Z$ be as in
\eqref{eq:Zform}. Then $Z$ cannot be in the interior of $\SEP(n+1,m+1)$. 
\end{lemma}

\begin{proof}
If $\rank(U) = 2$, it must be that $\tr(X) = \tr(Y) = 1$ \cite{Pataki.1998}. Then $\rank(\Ubar) = 1$, where $\Ubar$ is the Schur complement matrix
\[
\Ubar = \begin{pmatrix} X - xx\tran & V\tran - xy\tran \\ V-yx\tran & Y - yy\tran\end{pmatrix}.
\]
Therefore $\Ubar = (u ;v)(u ;v)\tran$, where $u\in\Rbb^m$ and $v\in\Rbb^n$, so $X-xx\tran = uu\tran$, $Y - yy\tran = vv\tran$, $V-yx\tran = vu\tran$. Since
$\tr(X) = \tr(Y) = 1$ we have $\norm{u}^2 = \tr(X -xx\tran) = 1-\norm{x}^2 $ and  $\norm{v}^2 = \tr(Y -yy\tran) = 1-\norm{y}^2 $.  Then
\[
\norm{V-yx\tran}_* = \norm{vu\tran}_* = \norm{u}\norm{v} = \sqrt{(1-\norm{x}^2)(1-\norm{y}^2)},
\]
and by Corollary \ref{cor:SEP_boundary}, if $Z\in\SEP(n+1,m+1)$ then $Z$ must be on the boundary of $\SEP(n+1,m+1)$.
\end{proof}

Lemma \ref{lem:SEP_interior} suggests that if the solution of \eqref{eq:Shor_quadratic} is rank 2, then it is very likely that the constraint $Z\in\SEP(n+1,m+1)$ is violated and therefore adding this constraint can strictly tighten the relaxation.  Adding this constraint to the Shor relaxation \eqref{eq:Shor_quadratic} and using the representation from Proposition \ref{prop:SEP_LMI}
results in the problem
\begin{eqnarray}
&\min& \mprod{C}{U}\label{eq:SEP_quadratic}\\
&{\rm s.t.}& \tr{X} \le 1,\ \tr{Y}\le 1,\ U_{00}=1,\nonumber\\
& &Z = \begin{pmatrix} 1 & x\tran \\ y & V\end{pmatrix}=\Wcal^*(T),\ \mprod{T}{\SS} = 0\ \forall \SS \in \Acal(n)\Kprod \Acal(m),\nonumber\\
& &U\gesem 0,\  T\gesem 0,\nonumber
\end{eqnarray}
where $T\in S^{nm}$.   Similar to the terminology used for the bilinear
problem of the previous section, we refer to \eqref{eq:SEP_quadratic} as the ``Full SEP" problem and use ``Lazy SEP" to refer to the approach where the constraints $\mprod{T}{\SS} = 0\ \forall \SS \in \Acal(n)\Kprod \Acal(m)$ are added as needed.

We will also consider the possibility of adding another constraint to tighten \eqref{eq:Shor_quadratic} based on the methodology of \cite{Anstreicher.2017}.
The constraints that $\norm{x}\le 1$ and $\norm{y}\le 1$ can be expressed in the form $\Ihat_x\gesem 0$, $\Ihat_y\gesem 0$, where
\begin{equation}\label{eq:Ihat}
\Ihat_x = \begin{pmatrix} I &  x \\ x\tran & 1\end{pmatrix} \in S^{m+1},
\quad  \Ihat_y = \begin{pmatrix} I &  y \\ y\tran &1\end{pmatrix} \in S^{n+1}.
\end{equation}
Then $\Ihat_x\gesem 0$ and $\Ihat_y\gesem 0$ imply that $\Ihat_x\Kprod \Ihat_y\gesem 0$, where
\[
\Ihat_x\Kprod \Ihat_y = \begin{pmatrix} \Ihat_y & & &x_1\Ihat_y \\
 & \ddots & &\vdots\\
 & & \Ihat_y &  x_m\Ihat_y\\
 x_1\Ihat_y & \cdots &x_m\Ihat_y & \Ihat_y\\ \end{pmatrix}.
 \]
Next,  the vector  $x_j y$ that appears in $x_j\Ihat_y$ can be replaced by $V_j$, the $j$th column of $V$, using $V$ as a proxy for the rank-one matrix $yx\tran$. Doing so
 results in the LMI constraint $K(Z)\gesem 0$, where
 \begin{equation}\label{eq:KRON_quadratic}
 K(Z)  = \begin{pmatrix} \Ihat_y & & &H_1(Z) \\
 & \ddots & &\vdots\\
 & & \Ihat_y &  H_m(Z)\\
 H_1(Z) & \cdots &H_m(Z) & \Ihat_y\\ \end{pmatrix},\quad 
 H_j(Z) = \begin{pmatrix} x_j I & V_j \\ V_j\tran & x_j \end{pmatrix}, 
 \end{equation}
$ j=1,\ldots,m$. We will refer to the constraint $K(Z)\gesem 0$ as the KRON constraint since it is  based on the Kronecker-product methodology from \cite{Anstreicher.2017}. The next lemma shows that the KRON constraint $K(Z)\gesem 0$  is dominated by the constraint that $Z\in\SEP(n+1,m+1)$.

\begin{lemma}\label{lem:SEP_vs_KRON}
Suppose that  $Z$ as in \eqref{eq:Zform} has $Z\in\SEP(n+1,m+1)$. Then $K(Z)\gesem 0$.
\end{lemma}

\begin{proof}
By Lemma \ref{lem:Z_decomp} we know that $Z$ can be written as a convex combination of matrices of the form
\begin{equation}\label{eq:Z_decomp}
Z^i =\begin{pmatrix} 1 & (x^i)\tran\\
y^i & y^i(x^i)\tran \end{pmatrix},\quad i=1,\ldots,k
\end{equation}
where  $\norm{x^i}\le 1$, $\norm{y^i}\le 1$ for each $i$. But each such $Z^i$ satisfies $K(Z^i)\gesem 0$ by construction, and therefore $K(Z)\gesem 0$ as well.
\end{proof}

Note that the matrix $K(Z)\in S^N$ for $N =(n+1)(m+1)$, which is approximately the same size as the matrix $\Wcal(M)$ that arises in the representation
for $ \LOP(n+1,m+1)$ in Proposition \ref{prop:LOP_LMI}. The KRON constraint can be implemented much more efficiently
by using cuts in place of the explicit LMI $K(Z)\gesem 0$; see the Appendix for details.

We next consider computational results on instances of the problem
\eqref{eq:quadratic}. To do so, we generated 60 random max-distance
instances, ten each for dimensions $n=2,4,6,8,10,15$. That is, we generated instances of
\eqref{eq:quadratic} of the form
\[
\max\{\norm{(Ax+a)-(By+b)}^2\suchthat \norm{x}\le 1,\ \norm{y}\le 1\},
\]
with \(n=m\), so that \(\Ecal_x=\{Ax+a\suchthat\norm{x}\le 1\}\) and
\(\Ecal_y=\{By+b\suchthat\norm{y}\le 1\}\) are the two original ellipsoids whose
farthest pair of points is sought.  The maps \(A\) and \(B\) were generated as
\(U\Sigma V^\top\) with Haar-random orthogonal \(U,V\) and singular values log-uniform on \([1,5.0]\), and the centers \(a\) and \(b\) have independent normal
entries.  Because the Shor relaxation \eqref{eq:Shor_quadratic} is already exact on most max-distance instances, the instances reported here were selected by a rank-one screen, applied to seeds starting from zero.  The screen first solves \eqref{eq:Shor_quadratic} and discards the instance when the two largest eigenvalues of the optimal \(U\) satisfy \(\lambda_1/\lambda_2>10^4\). Such a solution is effectively rank one, in which case \((x,y)\) is itself a feasible point attaining the bound and the relaxation is exact.  For the remainder, the screen re-solves the Shor relaxation with the Shor objective held fixed up to solver tolerance and with a random linear objective. This exposes rank-one optima when the original solve returned a higher-rank point on a flat optimal face.  If this second solve is still not effectively rank one, the instance is retained. Such instances are rare; retaining 10 per size required scanning up to 16970 seeds.

We compared the Shor SDP relaxation \eqref{eq:Shor_quadratic}, Shor strengthened by the
KRON constraint \(K(Z)\succeq 0\), Lazy KRON, Full SEP \eqref{eq:SEP_quadratic}, and
Lazy SEP, where ``Lazy KRON" refers to the KRON constraint implemented using cuts as described in the Appendix.
Unlike the bilinear problem of Section 3, none of these relaxations is
guaranteed exact for the class \eqref{eq:quadratic}.  For each conic method, we extracted
\((x,y)\) from the first column of the matrix $U$ \eqref{eq:Shor_matrices}, verified the two
ball constraints to tolerance \(10^{-7}\), and evaluated the original
objective from \eqref{eq:quadratic} at that point.  This gives a method-specific feasible
upper bound; no Gurobi incumbent or point recovered by another method is used
to calculate the gap.  
Gurobi was run with explicit bounds \(-1\le x_j\le 1\), \(-1\le y_i\le 1\), the two ball constraints, and time limit \(\max\{5,3\,t_{\max}\}\), where \(t_{\max}\) is the maximum time used by the conic methods on that instance.  Lazy KRON adds a single violated constraint and Lazy SEP adds at most 50 violated constraints in each outer-approximation round, both using relative violation tolerance \(10^{-7}\). 

\begin{table}[htbp]
\centering
\caption{Median wall-clock times for max-distance instances.}
\vskip 10pt
\label{tab:maxdist_times}
\small
\begin{tabular}{c r r r r r r r r}
\toprule
$n$ & cases & Shor & KRON & Lazy KRON & Full SEP & Lazy SEP & Gurobi & Gurobi opt \\
\hline
2 & 10 & 1 ms & 3 ms & 4 ms & 3 ms & 3 ms & 55 ms & 10 \\
4 & 10 & 1 ms & 10 ms & 8 ms & 7 ms & 7 ms & 4.93 s & 5 \\
6 & 10 & 2 ms & 69 ms & 10 ms & 21 ms & 14 ms & 5.00 s & 0 \\
8 & 10 & 2 ms & 605 ms & 22 ms & 84 ms & 32 ms & 5.00 s & 0 \\
10 & 10 & 2 ms & 4.60 s & 38 ms & 398 ms & 58 ms & 13.81 s & 0 \\
15 & 10 & 4 ms & 611.65 s & 151 ms & 16.33 s & 251 ms & 1543.72 s & 0 \\
\bottomrule
\end{tabular}
\end{table}

The median times required by the different methods are given in
Table~\ref{tab:maxdist_times}.  In the table, the column \(n\) gives the
dimensions of \(y\) and \(x\).  The Gurobi opt column gives the
number of instances solved to global optimality by Gurobi within its adaptive
time limit. The Full KRON solve for one instance (\(n=15\), seed 278), used two MOSEK
threads after repeated memory-related termination with the default parallel
setting; its formulation and tolerances were unchanged.

Table~\ref{tab:maxdist_diagnostics} reports bound quality on the same
instances. For Shor, KRON, Lazy KRON, and SEP, the tight column counts
instances whose relative gap compared to the recovered feasible solution
is at most $10^{-5}$, while max gap gives the largest
such gap over the instances of that size. If \(z_{\rm lb}\) is the
computed lower bound and \(z_{\rm ub}\) is the recovered feasible upper
bound, this relative gap is
\begin{equation}\label{eq:rel_gap}
\frac{z_{\rm ub}-z_{\rm lb}}{\max \set{ 1,|z_{\rm lb}|, |z_{\rm ub}|}}.
\end{equation}
 For each $n$,  the number of instances solved  by Lazy KRON and Lazy SEP was
the same as the number solved by KRON and SEP, respectively. The Lazy KRON and Lazy SEP columns give the
median/maximum outer-approximation rounds, total added cuts, and the largest
final relative violation of the separated constraint family. 

Comparing the times in Table \ref{tab:maxdist_times} with the results for the bilinear problems in Table \ref{tab:bilinear_times},
the results for Full SEP are similar for problems of the same size, but the results for Lazy SEP are substantially better for the
larger instances of the max-distance problems.  The reason for this can be seen by comparing the diagnostics in Tables \ref{tab:maxdist_diagnostics} and \ref{tab:bilinear_diagnostics}.  In particular, Lazy SEP required substantially fewer rounds, and cuts, on the max-distance problems compared to the bilinear problems. The KRON relaxation was tight on a large majority of max-distance instances but left a gap on three problems, one each of size 4, 8 and 15. 
The time for KRON grows even faster than for Full SEP, but the time for Lazy KRON is somewhat less than for Lazy SEP on the larger problem sizes.
As in the case of the bilinear problems, these problems are difficult for Gurobi to solve to global optimality for all but the smallest sizes.

\begin{table}[htbp]
\centering
\caption{Bound quality and diagnostics for max-distance instances.}
\label{tab:maxdist_diagnostics}
\vskip 10pt
\begingroup
\setlength{\tabcolsep}{2.5pt}
\small
\begin{tabular}{c|c r|c r|c r|r r r|r r r}
\toprule
& \multicolumn{2}{c|}{Shor} & \multicolumn{2}{c|}{KRON} & \multicolumn{2}{c|}{SEP} & \multicolumn{3}{c|}{Lazy KRON} & \multicolumn{3}{c}{Lazy SEP} \\
$n$ & tight & max gap & tight & max gap & tight & max gap & rounds & cuts & max viol & rounds & cuts & max viol \\
\midrule
2 & 0 & $8.8\cdot 10^{-1}$ & 10 & $1.3\cdot 10^{-6}$ & 10& $2.7\cdot 10^{-6}$ & 5/5 & 4/4 & $8.3\cdot 10^{-8}$ & 1/1 & 0/0 & $1.5\cdot 10^{-8}$ \\
4 & 0 & $7.7\cdot 10^{-1}$ & 9& $2.2\cdot 10^{-5}$ & 10 & $7.3\cdot 10^{-7}$ & 6/9 & 5/8 & $6.1\cdot 10^{-8}$ & 1/1 & 0/0 & $5.9\cdot 10^{-8}$ \\
6 & 0 & $7.6\cdot 10^{-1}$ & 10 & $2.4\cdot 10^{-6}$ & 10 & $3.3\cdot 10^{-8}$ & 4/7 & 4/6 & $8.4\cdot 10^{-8}$ & 1/2 & 0/50 & $2.6\cdot 10^{-9}$ \\
8 & 0 & $7.3\cdot 10^{-1}$ & 9 & $2.3\cdot 10^{-1}$ & 10 & $5.2\cdot 10^{-8}$ & 6/24 & 6/23 & $10.0\cdot 10^{-8}$ & 1/2 & 0/50 & $1.1\cdot 10^{-8}$ \\
10 & 0 & $6.5\cdot 10^{-1}$ & 10 & $1.0\cdot 10^{-6}$ & 10 & $8.4\cdot 10^{-8}$ & 7/12 & 6/11 & $9.6\cdot 10^{-8}$ & 1/2 & 0/50 & $2.0\cdot 10^{-9}$ \\
15 & 0 & $5.5\cdot 10^{-1}$ & 9 & $1.4\cdot 10^{-1}$ & 10 & $1.3\cdot 10^{-7}$ & 8/20 & 7/19 & $9.0\cdot 10^{-8}$ & 1/4 & 0/150 & $1.0\cdot 10^{-8}$ \\
\bottomrule
\end{tabular}
\endgroup
\end{table}

In addition to the 60 problems with results given in Tables \ref{tab:maxdist_times} and \ref{tab:maxdist_diagnostics}, we considered an additional
30 larger max-distance problems generated in a similar manner.  We ran only Lazy KRON and Lazy SEP on these instances due to the excessive times
that would be required by the other methods.  Results on these problems are given in Table \ref{tab:maxdist_larger}.

\begin{table}[htbp]
\centering
\caption{Results for larger max-distance instances}
\label{tab:maxdist_larger}
\vskip 10pt
\begingroup
\setlength{\tabcolsep}{2.5pt}
\small
\begin{tabular}{c c|r c r|r c r|r r r|r r r}
\toprule
& &\multicolumn{3}{c|}{Lazy KRON} & \multicolumn{3}{c|}{Lazy SEP} & \multicolumn{3}{c|}{Lazy KRON} & \multicolumn{3}{c}{Lazy SEP} \\
$n$& cases& time &tight & max gap & time &tight & max gap & rounds & cuts & max viol & rounds & cuts & max viol \\
\midrule 
$20$ &10&294 ms   & 8 & $1.7\cdot 10^{-5}$ & 1.01 s    &10 & $1.7\cdot 10^{-7}$ & 4/9 & 3/8 & $4.7\cdot 10^{-8}$ & 1/2 & 0/50 & $2.8\cdot 10^{-8}$ \\
$25$ &10& 3.85 s    &9& $9.2\cdot 10^{-5}$ & 3.20 s   &10 & $9.5\cdot 10^{-8}$ & 3/7 & 2/6 & $6.7\cdot 10^{-8}$ & 1/1 & 0/0 & $1.6\cdot 10^{-9}$ \\
$30$ &10& 14.54 s     & 9 & $1.8\cdot 10^{-5}$ &9.17 s  &10 & $1.7\cdot 10^{-8}$ & 4/7 & 3/6 & $8.5\cdot 10^{-8}$ & 1/3 & 0/100 & $1.5\cdot 10^{-9}$ \\
\bottomrule
\end{tabular}
\endgroup
\end{table}

\subsection{The two-trust-region subproblem}

In this section we consider the two-trust-region subproblem (TTRS), also often referred to as the Celis-Dennis-Tapia (CDT) problem:
\begin{eqnarray*}
&\min& x\tran Q x + c\tran x\\
&{\rm s.t.}& \norm{x} \le 1,\ \norm{Ax+b}\le 1,
\end{eqnarray*}
where $A$ is a nonsingular $n\times n$ matrix and the matrix $Q$ is indefinite.  Let
\begin{equation}\label{eq:TTRS_matrices}
C = \frac{1}{2} \begin{pmatrix} 0       & c\tran  \\
                               c  & 2Q	 \end{pmatrix},\quad  
    B=   \begin{pmatrix} \norm{b}^2       & b\tran A  \\
                               A\tran b  & A\tran A \end{pmatrix},\quad                                                      
   U_x =   \begin{pmatrix} 1       & x\tran  \\
                               x  & X \end{pmatrix}.
 \end{equation}
Using \eqref{eq:TTRS_matrices}, the Shor relaxation for TTRS can be written
\begin{eqnarray}
&\min& \mprod{C}{U_x}\label{eq:Shor_TTRS}\\
&{\rm s.t.}& \mprod{I-E_{00}}{U_x} \le 1,\ \mprod{B}{U_x}\le 1,\nonumber\\
& &U_x\gesem 0,\ \mprod{E_{00}}{U_x}=1.\nonumber
\end{eqnarray}
It is well known that the Shor relaxation \eqref{eq:Shor_TTRS} may fail to be tight, unlike the Shor relaxation for the
trust-region subproblem (TRS) which is TTRS without the second ellipsoidal constraint $\norm{Ax+b}\le 1$.  Several methods have been devised
to attempt to tighten the Shor relaxation.  In \cite{Burer.Anstreicher.2013}, SOC-RLT constraints are added by using supporting
hyperplanes from the ball constraint $\norm{x}\le 1$ combined with the second ellipsoidal constraint. This approach is
further strengthened in \cite{Anstreicher.2017}, using the Kronecker product of the two SOC constraints.  Our intent here is
to use the $\SEP(n+1,n+1)$ cone to strengthen the Shor relaxation \eqref{eq:Shor_TTRS}.

For $x\in\Rbb^n$ let $y=Ax+b$.  Then $yx\tran = (Ax+b)x\tran$, so replacing $xx\tran$ with $X$ we can write
\begin{equation}\label{eq:Z_TTRS}
Z = \begin{pmatrix} 1 & x\tran \\ y & V\end{pmatrix} = \begin{pmatrix} 1 & x\tran \\ Ax+b & AX+bx\tran \end{pmatrix} 
=   GU_x,\ \mbox{where}\ G=\begin{pmatrix} 1 & 0   \\ b & A\end{pmatrix}.
\end{equation}
As with the problem \eqref{eq:quadratic}, the Pataki rank bound \cite{Pataki.1998} implies that an extreme point solution matrix $U_x$ in
\eqref{eq:Shor_TTRS} has rank at most two, and if $\rank(U_x)=2$ then it must be that both $\mprod{I-E_{00}}{U_x}=1$ and
$\mprod{B}{U_x} = 1$ in the solution of \eqref{eq:Shor_TTRS}.

\begin{lemma} \label{lem:SEP_interior_TTRS}
Suppose that $U_x$ as in \eqref{eq:TTRS_matrices} is an extreme point of \eqref{eq:Shor_TTRS} with $\rank(U_x) = 2$.  Let $Z$ be as in
\eqref{eq:Z_TTRS}. Then $Z$ cannot be in the interior of $\SEP(n+1,n+1)$. 
\end{lemma}

\begin{proof}
It is easy to show that the Shor relaxation \eqref{eq:Shor_TTRS} is equivalent to \eqref{eq:Shor_quadratic} with $P=0$, $R=0$ and the added conditions $y=Ax+b$, $V=AX+bx\tran$ and $Y= AXA\tran + bx\tran A\tran + Axb\tran + bb\tran$.  In particular, for 
$y$, $V$ and $Y$ satisfying those conditions,
\[
\begin{pmatrix} 1 & 0 \\ 0 & I \\b & A\end{pmatrix}
\begin{pmatrix} 1 & x\tran \\ x & X \end{pmatrix} 
\begin{pmatrix} 1 & 0 & b\tran \\ 0 & I & A\tran\end{pmatrix} =
\begin{pmatrix} 1 &x\tran & y\tran \\ x & X & V\tran\\ y & V & Y\end{pmatrix},
\]
so $U\gesem 0$ in \eqref{eq:Shor_quadratic} if and only if $U_x\gesem 0$ in \eqref{eq:Shor_TTRS}, and the rank of $U_x$ in
\eqref{eq:Shor_TTRS} is equal to the rank of $U$ in \eqref{eq:Shor_quadratic}. The proof then follows the
proof of Lemma \ref{lem:SEP_interior} with $m=n$.
\end{proof}

Lemma \ref{lem:SEP_interior_TTRS} suggests that if $\rank(U_x)=2$ in the solution of \eqref{eq:Shor_TTRS}, then it is very likely that the constraint $Z=GU_x\in\SEP(n+1,n+1)$ is violated, and therefore adding this constraint can strengthen the relaxation.
Adding this constraint, using the representation from Proposition \ref{prop:SEP_LMI},  we obtain the problem
\begin{eqnarray}
&\min& \mprod{C}{U_x}\label{eq:SEP_TTRS}\\
&{\rm s.t.}& \mprod{I-E_{00}}{U_x} \le 1,\ \mprod{B}{U_x}\le 1,\ \mprod{E_{00}}{U_x}=1,\nonumber\\
& &GU_x=\Wcal^*(T),\ \mprod{T}{\SS} = 0\ \forall \SS \in \Acal(n)\Kprod \Acal(n),\nonumber\\
& &U_x\gesem 0,\  T\gesem 0.\nonumber
\end{eqnarray}
As with other problems that we have considered, we will refer to \eqref{eq:SEP_TTRS} as the Full SEP problem and use Lazy SEP to refer to the approach where the constraints $\mprod{T}{\SS} = 0\ \forall \SS \in \Acal(n)\Kprod \Acal(n)$ are added as needed. 

We will also consider the use of the Kronecker product constraint from \cite{Anstreicher.2017} to strengthen the Shor relaxation \eqref{eq:Shor_TTRS}.  The Kronecker product constraint is based on using the SOC constraints $\Ihat_x\gesem 0$,
$H(x)\gesem 0$, where
\[
\Ihat_x = \begin{pmatrix} I & x \\ x\tran & 1\end{pmatrix},\quad H(x) = \begin{pmatrix} I & Ax+b \\
x\tran A\tran +b\tran & 1\end{pmatrix}.
\]
Forming the Kronecker product $\Ihat_x\Kprod H(x)$ and substituting $X$ for $xx\tran$ results in the constraint
$K(U_x)\gesem 0$, where
\begin{eqnarray*}
 K(U_x)  &=& \begin{pmatrix} H(x) & & &H_1(U_x) \\
 & \ddots & &\vdots\\
 & & H(x) &  H_n(U_x)\\[10pt]
 H_1(U_x) & \cdots &H_n(U_x) & H(x)\\ \end{pmatrix},\\ 
 H_j(U_x) &=& \begin{pmatrix} x_j I & AX_j+bx_j \\ X_j\tran A\tran+x_j b\tran & x_j \end{pmatrix},\ j=1,\ldots,n,
 \end{eqnarray*}
and $X_j$ is the $j$th column of $X$. We will refer to $K(U_x)\gesem 0$ as the KRON constraint. It is shown in 
\cite{Anstreicher.2017} that the KRON constraint can be implemented much more efficiently using cuts as opposed to using the LMI $K(U_x)\gesem 0$.
Recall that for the quadratic problem \eqref{eq:quadratic}, we proved in Lemma \ref{lem:SEP_vs_KRON} that the
SEP constraint implied the KRON constraint.  Unfortunately we cannot extend this result to the TTRS problem, for the following 
reason.  The proof of Lemma \ref{lem:SEP_vs_KRON} uses the fact that from Lemma \ref{lem:Z_decomp}, $Z$ can be
written as a convex combination of matrices as in \eqref{eq:Z_decomp}. This fact remains true for the TTRS problem, using
$Z$ as in \eqref{eq:Z_TTRS}. However, 
we would need to have $y^i=Ax^i+b$ for each $i$ to argue that the KRON constraint is satisfied, and we have no
control over the $(x^i,y^i)$ in \eqref{eq:Z_decomp}.

\begin{table}[htbp]
\centering
\caption{Results on TTRS instances from \cite{Burer.Anstreicher.2013}}
\label{tab:TTRS_results}
\vskip 10pt
\begingroup
\setlength{\tabcolsep}{3pt}
\small
\begin{tabular}{c r|r r r| r r r|r r r }
\toprule
\multicolumn{2}{c|}{} & \multicolumn{3}{c|}{Full SEP} & \multicolumn{3}{c|}{Lazy SEP}& \multicolumn{3}{c}{Gurobi} \\
\(n\) & cases & time & tight & max gap & time & tight & max gap   & time&tight&max gap \\
\midrule
5 & 38 & 13 ms & 38 & $4.9\cdot 10^{-7}$ & 25 ms & 38 & $1.2\cdot 10^{-6}$    & 49 ms & 38 & $8.7\cdot 10^{-7}$ \\
10 & 70 & 457 ms & 70 & $2.3\cdot 10^{-7}$ & 144 ms & 70 & $2.0\cdot 10^{-7}$   &447 ms & 69 & $1.5\cdot 10^{-4}$ \\
20 & 104 & 1016.15 s & 104 & $3.5\cdot 10^{-7}$ & 2.18 s & 104 & $7.5\cdot 10^{-7}$ &487.45 s & 93 & $1.5\cdot 10^{-4}$ \\
\bottomrule
\end{tabular}
\endgroup
\end{table}

\begin{table}[htbp]
\centering
\caption{Additional diagnostics for Lazy SEP on TTRS instances from \cite{Burer.Anstreicher.2013}}
\label{tab:TTRS_diagnostics}
\vskip 10pt
\begingroup
\begin{tabular}{c|r r r}
\toprule
$n$& rounds & cuts&max viol\\ 
\midrule
5 & 2/2 & 80/100 & $7.8\cdot 10^{-9}$ \\
 10 & 2/2 & 100/100 & $8.4\cdot 10^{-9}$\\
20 & 2/3 & 100/200 & $3.9\cdot 10^{-8}$\\
\bottomrule
\end{tabular}
\endgroup
\end{table}

To computationally evaluate the effect of using the SEP constraint to strengthen the Shor relaxation \eqref{eq:Shor_TTRS} we
will first consider a set of test problems that were created in \cite{Burer.Anstreicher.2013}.  In particular we will use the subset of
212 problems that were not solved to optimality using the methodology based on SOC-RLT cuts in \cite{Burer.Anstreicher.2013}.
The 212 problems consist of 38 instances with $n=5$, 70 instances with $n=10$ and 104 instances with $n=20$ (see
\cite[Section 3]{Anstreicher.2017} for more details) and
have been used in a number of subsequent papers
\cite{Anstreicher.2017,Anstreicher.2024,Burer.2025,Consolini.Locatelli.2023,Yang.Burer.2016}.  In the computational
results of \cite[Section 3]{Anstreicher.2017}, 127 of the 212 problems were solved using the KRON constraint implemented using cuts to strengthen the
Shor relaxation. All but one of these problems was solved in \cite{Consolini.Locatelli.2023} using a tailored lower bound combined with local search, and all 212 problems were solved in \cite{Anstreicher.2024} and \cite{Burer.2025}.  The lifting introduced in
\cite{Burer.2025} is proved exact in the case that both constraints in TTRS are spherical, so $A=\alpha I$ for some $\alpha>0$.  

In Table \ref{tab:TTRS_results} we give the results of applying SEP, Lazy SEP and Gurobi to the 212 instances from \cite{Burer.Anstreicher.2013}.  All problems
were solved successfully by both conic methods.
The column \(n\) is the dimension and cases is the number of instances of that size. Time reports the median wall-clock solve time, 
tight counts method-specific recovered feasible upper bounds agreeing with their relaxation lower bounds to relative tolerance \(10^{-5}\) 
using \eqref{eq:rel_gap} and max gap is the largest such relative gap. 
Additional diagnostics for Lazy SEP reported in Table \ref{tab:TTRS_diagnostics} are median/maximum rounds, median/maximum cuts, and maximum final separated violation.  The dramatic time increase for Full SEP on the $n=20$ problems is notable, compared to the much more moderate increase required for Lazy SEP. 

In addition to the 212 problems from \cite{Burer.Anstreicher.2013}, we generated an additional 30 TTRS problems, ten each of sizes $n=20$, 25 and 30, that were filtered to ensure that the Shor relaxation was not tight. These problems were solved using only Lazy KRON and Lazy SEP due to the excessive times
required by the other methods. Results for these instances are given in Table \ref{tab:TTRS_larger}.

\begin{table}[htbp]
\centering
\caption{Results for additional TTRS instances}  
\label{tab:TTRS_larger}
\vskip 10pt
\begingroup
\setlength{\tabcolsep}{2pt}
\small
\begin{tabular}{c c|r c r|r c r|r r r|r r r}
\toprule
& &\multicolumn{3}{c|}{Lazy KRON} & \multicolumn{3}{c|}{Lazy SEP} & \multicolumn{3}{c|}{Lazy KRON} & \multicolumn{3}{c}{Lazy SEP} \\
\(n\)& cases & time &tight & max gap & time& tight & max gap & rounds & cuts & max viol & rounds & cuts & max viol \\
\midrule
20 & 10&   309 ms  &4 & $4.6\cdot 10^{-5}$ & 1.01 s   & 10 & $1.6\cdot 10^{-7}$ & 4/18 & 4/17 & $8.8\cdot 10^{-8}$ & 1/5 & 0/100 & $7.5\cdot 10^{-8}$ \\
25 & 10&    1.47 s &6 & $3.8\cdot 10^{-1}$ &   3.69 s   &  10& $3.8\cdot 10^{-6}$ & 7/22 & 6/21 & $7.8\cdot 10^{-8}$ & 1/10 & 0/900 & $9.9\cdot 10^{-8}$ \\
30 & 10&     20.14 s &4 & $4.2\cdot 10^{-1}$ &   12.15 s &   10 & $5.4\cdot 10^{-7}$ & 6/37 & 5/36 & $9.6\cdot 10^{-8}$ & 2/6 & 52/500 & $9.2\cdot 10^{-8}$ \\
\bottomrule
\end{tabular}
\endgroup
\end{table}

Although the addition of the SEP constraint to the Shor relaxation solves all of the problems considered above, we have determined that it
is not sufficient to give an exact representation for TTRS. To see this, consider the instance of TTRS with $n=2$, $A=\diag(a)$, and
rational data
\begin{equation}
    a=\frac{1}{12}\begin{bmatrix} 16\\ 13\end{bmatrix},
    \qquad
    b=\frac{1}{20}\begin{bmatrix} 11\\ -1\end{bmatrix},
    \qquad
    Q=-\frac{1}{16}\begin{bmatrix} 24 & 1\\ 1 & 13\end{bmatrix},
    \qquad
    c=\begin{bmatrix} -1\\ 0\end{bmatrix}.
    \label{eq:ttrs_counterexample_data}
\end{equation}
Here \(Q\) is negative definite, with eigenvalues \(-1.505635621484\) and
\(-0.806864378516\), so every optimal solution lies on the boundary of the
feasible region.  Because \(\norm{b}<1\), the two balls have a common interior
point and the problem is strictly feasible.  Note that multiplying \(Q\) and
\(c\) by \(16\) makes the objective data integral and rescales all of the objective
values below by the same factor.

The Shor relaxation of the instance \eqref{eq:ttrs_counterexample_data} has solution value \(z_{\rm Shor}=-0.768293943469\),
while addition of the SEP constraint results in a solution value  $z_{\SEP}=-0.632844508289$.
Since \(n=2\), the optimal value of \eqref{eq:ttrs_counterexample_data} can be certified
directly by enumerating the stationary points and mutual intersections of the
two ellipse boundary arcs, which yields
\[
    \zstar=-0.546797627007,
    \qquad
    x^*=\begin{bmatrix} 0.185929524305\\ 0.602567865183\end{bmatrix},
\]
with the second ellipsoidal constraint active and the first slack at \(x^*\).  A
global solve using Gurobi confirms this value to nine digits. The SEP relaxation therefore has a gap of
$\zstar-z_{\SEP}=0.086046881282$, or a relative gap of 15.74\%.

We also considered solving the instance \eqref{eq:ttrs_counterexample_data} using the 
\emph{beta relaxation} from \cite{Burer.2025}.  The beta relaxation
introduces variables \(\beta_i\) representing \(x_i^2\), so that the
two ellipsoid constraints become the linear inequalities
\[
    1-\sum_{i=1,2}\beta_i\geq0,
    \qquad
    1-\norm{b}^2-2\sum_{i=1,2} a_ib_ix_i-\sum_{i=1,2} a_i^2\beta_i\geq0.
\]
The relaxation then imposes the conditions \(\beta_i\geq x_i^2\) through rotated second-order cones, and adds the
Shor, RLT, SOC-RLT, and Kronecker constraints on the resulting lift.
On the instance \eqref{eq:ttrs_counterexample_data}, the beta relaxation has solution value
 $z_{\beta}=-0.550899217259$, with a gap  of $\zstar-z_{\beta}=0.004101590252$ or relative gap of 0.75\%. To our knowledge, this is the first reported TTRS instance on which the beta relaxation fails to be exact. 

\subsection{A noxious location problem}

Given points \(p_1,\ldots,p_m\in\Rbb^2\), consider the planar location problem
\begin{equation}
\begin{array}{ll}
\max & \theta \\
\text{s.t.} & \|x-p_i\|\geq\theta,\qquad i=1,\ldots,m,\\
&x\in\conv\{p_1,\ldots,p_m\}.
\end{array}
\label{eq:noxious}
\end{equation}
We assume without loss of generality that the points \(p_1,\ldots,p_m\in\Rbb^2\) lie in the unit disk.  For an arbitrary set of points $\set{p_i}$, we can if necessary find the minimum radius disk
containing the points by solving a convex optimization problem, and then translate and/or scale the points so that they lie in the unit disk.  We assume that a hyperplane description of $\conv\set{p_1,\ldots,p_m}$ is given by  $\set{x\suchthat Ax\le b}$ where $A$ is a $k\times 2$ matrix.

Let \(u=(x_1,x_2,\theta,\sigma)^T\), where \(\sigma\geq0\), and introduce
\[
H=\begin{pmatrix}1&u^T\\u&U\end{pmatrix}\succeq0.
\]
In order to convexify the constraints of \eqref{eq:noxious} we use the approach of \cite{Burer.Dong.2012} to place a subset of the variables on the surface of a sphere. To accomplish this we add the constraint
\begin{equation}
U_{11}+U_{22}+U_{44}=1
\label{eq:sphere}
\end{equation}
as a proxy for the sphere constraint $\norm{x}^2 + \sigma^2=1$.
Note that the variable $\theta$ (or $U_{33}$, the proxy for $\theta^2$) does not enter into the sphere equation \eqref{eq:sphere}.  The constraints for the Shor
relaxation are then
\begin{equation}
U_{11}+U_{22}-2p_i^Tx+\|p_i\|^2\geq U_{33},
\qquad i=1,\ldots,m,
\label{eq:shor-distance}
\end{equation}
together with \(Ax\leq b\), \(\sigma\geq0\), \eqref{eq:sphere}, and
\(H\succeq0\).  Using  \(\sigma\) and $U_{44}$ in this way is equivalent to the usual
Shor inequality \(U_{11}+U_{22}\leq1\).

In order to apply constraints based on LOP or SEP cones we need to express the constraints of \eqref{eq:noxious} using Lorentz cones. To accomplish this,
for \(i=1,\ldots,m\), define
\[
r_i=\sqrt{1+\|p_i\|^2},\qquad \bar p_i=p_i/r_i,
\]
and let
\[
z_0=(1,x_1,x_2,\sigma)^T,
\qquad
z_i=(r_i-\bar p_i^Tx,\ \bar p_i^Tx,\ \theta,\ \sigma)^T.
\]
Then \(z_0\in L_4\), and
\begin{align*}
z_i^T\operatorname{diag}(1,-1,-1,-1)z_i
 &=1+\|p_i\|^2-2p_i^Tx-\theta^2-\sigma^2\\
 &=\|x-p_i\|^2-\theta^2,
\end{align*}
where the second equality uses the sphere constraint \(\|x\|^2+\sigma^2=1\).  Thus
\(z_i\in L_4\) is exactly the reverse-distance constraint at a point with $\norm{x}^2 + \sigma^2 =1$.
Next, let \(w=(1,u^T)^T\) and choose \(G_i\) so that \(z_i=G_iw\).  At a rank-one
point,
\[
z_i z_j^T=G_iww^TG_j^T.
\]
Replacing \(ww^T\) by \(H\) gives the valid pairwise constraints
\begin{equation}
G_iHG_j^T\in\SEP(4,4),
\qquad 0\leq i<j\leq m.
\label{eq:sep}
\end{equation}
There are \(m+1\) Lorentz blocks and \(m(m+1)/2\) such pairs.  In the
polynomial representation of \(\SEP(4,4)\), each pair uses a \(9\times9\)
positive semidefinite auxiliary matrix and nine equations.
In our computational tests we compare the following relaxations.

\begin{itemize}
\item \textbf{Shor} consists of \eqref{eq:sphere},
      \eqref{eq:shor-distance}, \(Ax\leq b\), \(\sigma\geq0\), and
      \(H\succeq0\).
\item \textbf{RLT} adds the lifted products of the hull slacks
      \(b_j-a_j^Tx\), the first-moment constraints \(z_i\in L_4\), and
      the SOC--RLT constraints obtained by multiplying each hull slack by
      \(z_0\in L_4\).
\item \textbf{KRON} adds the Kronecker positive semidefinite constraint for
      every distinct pair among \(z_0,\ldots,z_m\).
\item \textbf{Full SEP} replaces each KRON constraint by
      \eqref{eq:sep} using the complete polynomial SEP representation.
\item \textbf{Lazy SEP} uses the same SEP representation but initially omits
      the equations and adds violated coordinate equations as needed.
\end{itemize}

For any instance, the true solution value of \eqref{eq:noxious} can be
computed exactly by finite enumeration.  Let \(P=\conv\{p_1,\ldots,p_m\}\)
and, for each \(i\), let \(C_i=\{x\suchthat\norm{x-p_i}\le\norm{x-p_j}\
\forall j\}\) be the Voronoi cell of \(p_i\).  The cells \(C_i\) are
polyhedral and cover the plane, so \eqref{eq:noxious} is equivalent to
\(\max_i \max\{\norm{x-p_i}\suchthat x\in P\cap C_i\}\).  Each inner
problem maximizes a convex function over a polygon, so its maximum is
attained at a vertex of \(P\cap C_i\).  Every such vertex is one of: a
vertex of \(P\); the intersection of an edge of \(P\) with the perpendicular
bisector of some pair \(p_j,p_k\); or a Voronoi vertex, that is, the
circumcenter of some triple \(p_j,p_k,p_l\).  We therefore enumerate
these \(O(m^3)\) candidate points, discard those outside \(P\), evaluate
\(\min_i\norm{x-p_i}\) at each remaining candidate, and take the largest
value.  The candidate set is a superset of the vertices of the polygons
\(P\cap C_i\), since it includes, for example, the intersection of an edge
of \(P\) with every bisector rather than only those bisectors that bound
a cell along that edge.  This causes no harm as every candidate is a
feasible point of \eqref{eq:noxious}, so the largest value of
\(\min_i\norm{x-p_i}\) over the candidates is the true optimal value. The bound ordering
\[
\theta_{\rm true}\leq\theta_{\rm Full\ SEP}
\leq\theta_{\rm KRON}\leq\theta_{\rm RLT}\leq\theta_{\rm Shor}
\]
held on every completed instance, and Full SEP and Lazy SEP differed by at most
\(3.5\times10^{-10}\) on the instances where both were run.

\begin{table}
\caption{Distance upper bounds for regular \(m\)-gons}
\vskip 10pt
\label{tab:regular}
\centering
\begin{tabular}{r|rrrr}
\toprule
\(m\) & Shor & RLT & KRON & Full SEP \\
\midrule
3  & 1.414214 & 1.414214 & 1.174750 & 1.114373 \\
4  & 1.414214 & 1.414214 & 1.133203 & 1.001735 \\
5  & 1.414214 & 1.414214 & 1.163184 & 1.025778 \\
6  & 1.414214 & 1.414214 & 1.133203 & 1.001735 \\
7  & 1.414214 & 1.414214 & 1.149584 & 1.010315 \\
8  & 1.414214 & 1.414214 & 1.133203 & 1.001735 \\
10 & 1.414214 & 1.414214 & 1.133203 & 1.001735 \\
12 & 1.414214 & 1.414214 & 1.133203 & 1.001735 \\
16 & 1.414214 & 1.414214 & 1.133203 & 1.001735 \\
\bottomrule
\end{tabular}
\end{table}

We first consider $m\ge 3$ points corresponding to vertices of a
regular polygon on the unit circle. For any $m$ it is easy to see
that the true solution value is one, attained at the origin. In Table
\ref{tab:regular} we give the solution values for the different
relaxations for various $m$. Full SEP displays the strongest bound
consistently. Lazy SEP is less attractive for this fixed-cone-size
family. At \(m=16\), it adds 1,218 of the 1,224 available coordinate
equations and takes about 35 seconds, whereas Full SEP takes under two
seconds.

We next consider instances where $m$ points are first uniformly
chosen from the unit disk and then translated and scaled using
their minimum enclosing disk. Ten instances were solved for
each $m$. Table \ref{tab:random} reports additive gaps in the
distance objective, \(\theta_{\rm ub}-\theta_{\rm true}\).
The median Full SEP times for \(m=4,6,8,10,15,20\) were
\(0.046,0.096,0.174,0.313,0.819,2.213\) seconds, respectively, while the
median KRON times were \(0.074,0.184,0.506,1.430,7.235,35.145\) seconds. Clearly SEP
gives a substantial improvement over the Shor bound on these problems,
but the relative improvement compared to the Shor gap appears to be
decreasing with $m$.

\begin{table}
\caption{Results for normalized random instances}
\label{tab:random}
\vskip 10pt
\centering
\begin{tabular}{rr|c|c|c|c}
\toprule
 & &Shor gap &  RLT gap & KRON gap & SEP gap\\ 
\(m\) & cases & median/max & median/max & median/max& median/max\\
\midrule
4  & 10 & 0.462/0.525 & 0.454/0.525 & 0.152/0.234 & 0.084/0.219 \\
6  & 10 & 0.445/0.495 & 0.426/0.494 & 0.203/0.264 & 0.144/0.257 \\
8  & 10 & 0.480/0.506 & 0.472/0.505 & 0.268/0.355 & 0.217/0.351 \\
10 & 10 & 0.526/0.618 & 0.526/0.607 & 0.318/0.389 & 0.249/0.332 \\
15 & 10 & 0.537/0.624 & 0.537/0.624 & 0.358/0.402 & 0.315/0.387 \\
20 & 10 & 0.581/0.613 & 0.581/0.613 & 0.402/0.430 & 0.365/0.402 \\
\bottomrule
\end{tabular}
\end{table}

The above results show that in general the SEP constraints do not fully close the gap in instances of \eqref{eq:noxious}.  However, there are cases
where the SEP constraints do close the gap to nearly zero.  For one such example, consider the four points
\begin{equation}\label{eq:targeted}
\begin{aligned}
p_1&=(-0.6188988868,\phantom{-}0.1752460829),\quad
p_2=(\phantom{-}0.6115320348,-0.5806915227),\\
p_3&=(\phantom{-}0.5533450604,-0.7177653858),\quad
p_4=(-0.6230854841,\phantom{-}0.1902842182).
\end{aligned}
\end{equation}
The exact solution value for \eqref{eq:noxious} with the points \eqref{eq:targeted} is \(0.7256779914\). In Table \ref{tab:targeted} we give the bound values obtained by various relaxations for this instance. Full SEP leaves an additive distance gap of approximately \(1.20\times10^{-4}\) compared
to a gap of over $0.49$ for the Shor relaxation.

\begin{table}\caption{Bounds for instance with points \eqref{eq:targeted}}
\label{tab:targeted}
\centering
\vskip 10pt
\begin{tabular}{r|r}
\toprule
method & distance upper bound \\
\midrule
Shor     & 1.2174074784 \\
RLT      & 1.0382899391 \\
KRON     & 0.8379375081 \\
Full SEP & 0.7257975088 \\
\midrule
exact value & 0.7256779914 \\
\bottomrule
\end{tabular}

\end{table}

\section{Conclusion}

In this paper we have demonstrated how Hildebrand's  polynomial LMI representations of the Lorentz positive cone $ \LOP(n+1,m+1)$ and its dual the Lorentz separable cone $\SEP(n+1,m+1)$ can be used to tighten SDP relaxations of nonconvex quadratic optimization problems. We believe that additional applications of these LMI representations could be a fruitful area for further research.

\section*{Acknowledgement}

The authors are grateful to Roland Hildebrand for helpful communications including the details of \cite{Nemirovski.2005}. The AI assistants Claude (Anthropic) and Codex (OpenAI) were used to write the code for the computational experiments and for light editing of the text.

\bibliographystyle{spmpsci}
\bibliography{Lorentz}

\section*{Appendix}

In this appendix we consider the implementation of the Kronecker product constraint $K(Z)\gesem 0$, where $K(Z)$ from \eqref{eq:KRON_quadratic}
applies to the quadratic optimization problem \eqref{eq:quadratic}. The methodology used here is based on \cite{Anstreicher.2017} but omits some
unneeded details.

Let $H(y)=\Ihat_y$ as in \eqref{eq:Ihat}. We assume that $\norm{y}< 1$, so $H(y)$ is nonsingular.  If $\norm{y}=1$ then $Y=yy\tran$ in the Shor relaxation. If in addition $\norm{x}=1$ then $X=xx\tran$ and $(x,y)$ are
an optimal solution to the original problem \eqref{eq:quadratic}.  If $\norm{y} = 1$ and $\norm{x}<1$ then we can interchange  $x$ and $y$. Let
\[
W(Z)  = \begin{pmatrix} I & & & \\
 & \ddots & & \\
 & & I &  \\
 -H_1(Z)H(y)^{-1} & \cdots &-H_m(Z)H(y)^{-1}&I\\ \end{pmatrix}.
\]
Then
\begin{eqnarray*}
W(Z)K(Z)W(Z)\tran &=& \begin{pmatrix} 
H(y) & & & \\
 & \ddots & &\\
    & & H(y) & \\
& & &\Kprime(Z) \end{pmatrix},\\
 \Kprime(Z) &=& H(y) - \sum_{j=1}^m H_j(Z)H(y)^{-1}H_j(Z),
\end{eqnarray*}
 and $K(Z)\gesem 0$ if and
only if $\Kprime(Z)\gesem 0$.  Suppose on the contrary that for $Z=\Zbar$ there is an  $a\in\Rbb^{n+1}$ with $a\tran \Kprime(\Zbar) a<0$.  Then $b\tran K(\Zbar) b <0$, where
\[
b= W(\Zbar)\tran \begin{pmatrix} 0 \\ \vdots \\ 0 \\a \end{pmatrix} =
\begin{pmatrix} B_1\\ \vdots \\ B_m \\ a 
\end{pmatrix},
\]
with $B_j= -H(\ybar)^{-1} H_j(\Zbar)a$, $j=1,\ldots,m$. We consider $B_j$ to be the $j$th column of an $(n+1)\times m$ matrix $B$. Therefore $b\tran K(Z)b\ge 0$ is a valid
constraint on $Z$ that is violated at $\Zbar$, where
\begin{equation}\label{eq:quadratic_cut}
b\tran K(Z) b= a\tran H(y)a +  \sum_{j=1}^m  \big(B_j\tran H(y)B_j + 2a\tran H_j(Z)B_j\big).
\end{equation} 
Let $a = (\abar; \alpha)$ and $B=(\Bbar; \beta\tran)$ where $\beta\in\Rbb^m$. Then
\begin{eqnarray*}
a\tran H(y)a & = & \norm{\abar}^2 + 2\alpha \abar\tran y + \alpha^2\\
B_j\tran H(y)B_j &=& \norm{\Bbar_j}^2 + 2\beta_j\Bbar_j\tran y + \beta_j^2\\
a\tran H_j(Z)B_j & =& \abar\tran \Bbar_j x_j  + \beta_j\abar\tran V_j + \alpha \Bbar_j\tran V_j
+\alpha\beta_jx_j.
\end{eqnarray*} 
Substituting terms into \eqref{eq:quadratic_cut}, we obtain a valid linear constraint, or cut of the form $\mprod{C}{Z}\ge 0$, with $\mprod{C}{\Zbar}< 0$. More precisely,
\begin{eqnarray*}
\mprod{C}{Z} &=& \big(\norm{a}^2 + \sum_{j=1}^m \norm{B_j}^2\big) +2\sum_{j=1}^m (a\tran B_j)x_j \\
 &&+2\big(\alpha \abar\tran + \sum_{j=1}^m \beta_j\Bbar_j\tran\big) y
+2\sum_{j=1}^m (\beta_j\abar\tran + \alpha \Bbar_j\tran) V_j.
\end{eqnarray*}

\end{document}